\documentclass{article}

\usepackage[a4paper]{geometry}		
\usepackage[english]{babel}			
\usepackage[utf8]{inputenc}			
\usepackage[a-1b]{pdfx}			
\usepackage{fancyhdr}
\usepackage{comment}
\usepackage{enumerate}

\usepackage{graphicx}				
\usepackage{hologo}					

\usepackage{amssymb,amsmath,amsthm,bbm} 
\usepackage{mathrsfs}
\usepackage{listings}	
\usepackage{mathtools}

\hypersetup{colorlinks,allcolors=blue}

\newtheorem{Theorem}{Theorem}[section]
\newtheorem{Definition}{Definition}[section]
\newtheorem{Proposition}{Proposition}[section]
\newtheorem{Corollary}{Corollary}[section]

\newtheorem{Remark}{Remark}[section]

\newtheorem*{specialAssumption}{Assumption}
\numberwithin{equation}{section}

\makeatletter
\renewenvironment{proof}[1][\proofname] {\par\pushQED{\qed}\normalfont\topsep6\p@\@plus6\p@\relax\trivlist\item[\hskip\labelsep\itshape\bfseries#1\@addpunct{.}]\ignorespaces}{\popQED\endtrivlist\@endpefalse}
\makeatother

\def \N{\mathbb{N}}
\def \R{\mathbb{R}}

\def \E{\mathbb{E}}
\def \F{\mathbb{F}}

\def \P{\mathbb{P}}

\def \S{\mathbb{S}}

\def \Ac{{\cal A}}

\def \Cc{{\cal C}}

\def \Fc{{\cal F}}
\def \Gc{{\cal G}}
\def \Hc{{\cal H}}

\def \Lc{{\cal L}}

\def \Sc{{\cal S}}
\def \Tc{{\cal T}}

\DeclareMathOperator*{\esssup}{ess\,sup}

\title{On the optimal stopping problem for diffusions and an approximation result for stopping times}
\author{
Andrea COSSO\footnote{Universit\`a degli Studi di Milano; andrea.cosso@unimi.it} \quad\qquad
Laura PERELLI\footnote{Universit\`a degli Studi di Milano; laura.perelli@unimi.it\vspace{2mm}\\\textbf{Acknowledgements.} A. Cosso acknowledges support from GNAMPA-INdAM (of which L. Perelli is also member), the MUR project PRIN 2022 ``Entropy martingale optimal transport and McKean-Vlasov equations'', and the MUR project PRIN 2022 PNRR ``Probabilistic methods for energy transition''.}}
\date{March 4, 2025}

\allowdisplaybreaks

\begin{document}
\maketitle


\begin{abstract}
\noindent In this article, we study the classical finite-horizon optimal stopping problem for multidimensional diffusions through an approach that differs from what is typically found in the literature. More specifically, we first prove a key equality for the value function from which a series of results easily follow. This equality enables us to prove that the classical stopping time, at which the value function equals the terminal gain, is the smallest optimal stopping time, without resorting to the martingale approach and relying on the Snell envelope. Moreover, this equality allows us to rigorously demonstrate the dynamic programming principle, thus showing that the value function is the viscosity solution to the corresponding variational inequality. Such an equality also shows that the value function does not change when the class of stopping times varies. To prove this equality, we use an approximation result for stopping times, which is of independent interest and can find application in other stochastic control problems involving stopping times, as switching or impulsive problems, also of mean field type.
\end{abstract}

\vspace{5mm}


\vspace{5mm}


\section{Introduction}

In the present article, we consider the classical finite-horizon optimal stopping problem for a state process evolving according to the following stochastic differential equation:
\begin{equation}\label{SDE_Intro}
    \begin{cases}
        dX_s = b(s,X_s)ds + \sigma(s,X_s)dW_s, & s\in[t,T],\\
        X_t = x\in\R^d.
    \end{cases}
\end{equation}
Here $W=(W_s)_{s\in[t,T]}$ is an $m$-dimensional Brownian motion defined on a complete probability space $(\Omega,\Fc,\P)$, while the coefficients $b,\sigma\colon[0,T]\times\R^d\rightarrow\R^d,\R^{d\times m}$ satisfy the standard Lipschitz and linear growth conditions (for precise assumptions, see the first paragraph of Section \ref{S:Problem}). We denote by $X^{t,x}=(X_s^{t,x})_{s\in[t,T]}$ the solution to equation \eqref{SDE_Intro}. We let $\F=(\Fc_s)_{s\in[0,T]}$ be the $\P$-completion of the filtration generated by $W$. Moreover, for every $t\in[0,T]$, we let $\F^t=(\Fc_s^t)_{s\in[t,T]}$ be the $\P$-completion of the filtration generated by the Brownian increments $(W_s-W_t)_{s\in[t,T]}$. We then consider the reward functional in its classical form:
\[
    J(t,x,\tau) \ = \ \E\Bigg[\int_{t}^{\tau} f(s,X^{t,x}_s)ds \ + \ g(X^{t,x}_\tau)\Bigg], \quad \text{for all }\tau\in\Tc_{t,T},
\]
where $\Tc_{t,T}$ denotes the family of all $[t,T]$-valued stopping times with respect to $\F$. Given $\theta\in\Tc_{0,T}$, we also define $\Tc_{\theta,T}$ in a similar way. Notice that, for the optimal stopping problem starting at time $t$, it is reasonable to consider either all stopping times in $\Tc_{t,T}$ or to restrict to the subset $\tilde\Tc_{t,T}\subset\Tc_{t,T}$ consisting of all $\F^t$-stopping times. Therefore, we define the two value functions
\[
    v(t,x) \ = \ \sup_{\tau\in\Tc_{t,T}} J(t,x,\tau), \qquad\qquad \tilde v(t,x) \ = \ \sup_{\tau\in\tilde\Tc_{t,T}} J(t,x,\tau).
\]
We will prove that the functions $v$ and $\tilde v$ are indeed equal (Corollary \ref{C:v=v_tilde}). This is a consequence of the following key equality for $\tilde v$ (and hence for $v$), which corresponds to Theorem \ref{T:v_tilde} and is one of the main results of the paper: given $\theta\in\Tc_{0,T}$ and a $\Fc_\theta$-measurable random variable $\xi\colon\Omega\rightarrow\R^d$,
\begin{equation}\label{Vtilde_Intro}
    \tilde v(\theta,\xi) \ = \ \esssup_{\tau\in\Tc_{\theta,T}} \E\bigg[\int_{\theta}^{\tau} f(s,X^{\theta,\xi}_s)ds \ + \ g(X^{\theta,\xi}_\tau) \bigg|\Fc_\theta\bigg], \qquad \P\text{-a.s.}
\end{equation}
We provide a direct proof of this equality, based on subsequent approximations of both $\theta$ and $\xi$ by discrete random variables. This proof relies in particular on an approximation result for stopping times (Theorem \ref{T:Approx}), which is of independent interest and can be applied to other stochastic control problems involving stopping times, such as switching or impulsive problems, including those of mean field type (see for instance \cite{CossoPerelli_MeanField}, where it is applied to a class of mean field optimal stopping problems). From equality \eqref{Vtilde_Intro} we show (Theorem \ref{T:Optimal}) that the stopping time $\hat\tau^{t,x}\in\tilde\Tc_{t,T}$, given by
\begin{equation}\label{OptStoppTime_Intro}
\hat{\tau}^{t,x} \ = \ \inf\big\{s\in[t,T] \ \big| \ v(s,X^{t,x}_s) = g(X^{t,x}_s)\big\},
\end{equation}
is the smallest optimal stopping time for the problem that starts at time $t$ from state $x$. To prove this, we do not need to use the martingale approach to optimal stopping and to rely on the Snell envelope (see Remark \ref{R:Snell} for more details on this point). 
It is worth noting that the martingale approach has been adopted since the very beginning to the study of optimal stopping problems, see for instance \cite{ArrowBlackwellGirshick,Snell,ChowRobbinsSiegmund}. For an updated presentation, we refer to the book \cite{PeskirShiryaev}.

\noindent By equality \eqref{Vtilde_Intro} we also derive the dynamic programming principle for $v$ (Theorem \ref{T:DPP}): for every $\tau'\in\Tc_{t,T}$,
\[
v(t,x) \ = \ \sup_{\tau\in\Tc_{t,T}} \E\bigg[\int_{t}^{\tau\land\tau'} f(s,X^{t,x}_s)ds + g(X^{t,x}_\tau)\mathbbm{1}_{\{\tau<\tau'\}} + v(\tau',X^{t,x}_{\tau'})\mathbbm{1}_{\{\tau\geq\tau'\}}\bigg].
\]
As it is well-known, we can then show that $v$ is a viscosity solution to the following variational inequality (Theorem \ref{T: V tilde solution HJB}; we report the proof of this result for completeness):
\[
\begin{cases}
    \vspace{2mm}\max\big\{\partial_t v + \langle b,\partial_x v\rangle + \frac{1}{2}\text{tr}(\sigma\sigma^T \partial^2_x v) + f, g - v\big\}=0, & \qquad \text{on }[0,T)\times\R^d, \\
    v(T,x) = g(x), &\qquad x\in\R^d.
\end{cases}
\]
Although equality \eqref{Vtilde_Intro} is somehow known in the theory of optimal stopping (see for instance \cite[formulae (2.2.65)-(2.2.67)]{PeskirShiryaev}), to our knowledge a rigorous proof of it is still missing. 
This is perhaps because rigorous approaches are rarely developed in such specific contexts, being instead typically set in more general Markov settings where, for instance, the notions of excessive and superharmonic functions are used.
In this regard, see the following classical references: \cite{Shiryaev,ElKaroui,KaratzasShreve98,PeskirShiryaev}. Moreover, for optimal stopping of diffusions, see \cite{JackaLynn,Pedersen} and, in book form, \cite{Pham,TouziBook}. Regarding rigorous proofs of the dynamic programming principle for optimal stopping problems, see also the references \cite{BouchardTouzi,DeAngelisMilazzo,ElKarouiTanII,Krylov}. Finally, for the viscosity solutions approach to optimal stopping, we mention \cite{OksendalReikvam,Pham97}.\\ 
\\
\noindent The rest of the paper is organized as follows. In Section \ref{S:Problem}, we formulate the finite-horizon optimal stopping problem for a multidimensional diffusion. In particular, we prove the key equality for $\tilde v$ (Theorem \ref{T:v_tilde}) and show that $v$ and $\tilde v$ coincide (Corollary \ref{C:v=v_tilde}). We then show that $\hat\tau^{t,x}$ in \eqref{OptStoppTime_Intro} is the smallest optimal stopping time. Afterward, we derive the dynamic programming principle for $v$ (Theorem \ref{T:DPP}) and prove that $v$ is a viscosity solution to the Hamilton-Jacobi-Bellman equation (Theorem \ref{T: V tilde solution HJB}). 
Finally, in Section \ref{S:Approx}, we prove the approximation result for stopping times (Theorem \ref{T:Approx}).
\section{The optimal stopping problem}\label{S:Problem}
\textbf{Probabilistic framework and notations.} Let $(\Omega,\Fc,\P)$ be a complete probability space and let $T>0$. Consider an $m$-dimensional Brownian motion $W = (W^1_t,\dots,W^m_t)_{t\in[0,T]}$ and denote by $\F=(\Fc_t)_{t\in[0,T]}$ the $\P$-completion of its natural filtration. Moreover, for every $t\in[0,T]$, denote by $\F^t=(\Fc_s^t)_{s\in[t,T]}$ the $\P$-completion of the filtration of the Brownian increments $(W_s-W_t)_{s\in[t,T]}$. Notice that $\F^0$ coincides with $\F$. Given $t\in[0,T]$, we denote by $\Tc_{t,T}$ and $\tilde\Tc_{t,T}$ the set of all $[t,T]$-valued stopping times with respect to $\F$ and $\F^t$, respectively. For every $\theta\in\Tc_{0,T}$ we set $\Fc_\theta:=\{A\in\Fc_T|A\cap\{\theta\leq t\}\in\Fc_t,\;\forall\,t\in[0,T]\}$ and we denote by $\Tc_{\theta,T}$ the set of all $\F$-stopping times $\tau$ such that $\theta\leq\tau\leq T$. Given a sub-$\sigma$-algebra $\Gc$ of $\Fc$, we denote by $L^2(\Gc;\R^d)$ the set of square-integrable, $\Gc$-measurable, $d$-dimensional random variables. Finally, we denote by $\S^2$ the set of c\`adl\`ag, $d$-dimensional, $\F$-progressively measurable processes $X=(X_t)_{t\in[0,T]}$ such that
\[
    \|X\|_{ \S^2} \ := \ \E\bigg[\sup_{0\leq t\leq T}|X_t|^2\bigg]^{1/2} \ < \ +\infty.
\]
Let $b,\sigma,f\colon[0,T]\times\R^d\longrightarrow\R^d,\R^{d\times m},\R$ and $g\colon\R^d\longrightarrow\R$ be measurable functions satisfying the following assumptions, which will always be in force.
\begin{specialAssumption}[\textbf{A}]\label{A1}
\quad\begin{enumerate}[1)]
    \item $b,\sigma$ are Lipschitz continuous in $x$ uniformly in $t$: there exists $L>0$ such that, for all  $t\in[0,T]$, $x,y\in\R^d$,
    \[
        |b(t,x)-b(t,y)| \ + \ |\sigma(t,x)-\sigma(t,y)| \ \leq \ L\,|x-y|;
    \]
    \item there exists $K>0$ such that, for all $t\in[0,T]$,
    \[
        |b(t,0)| \ + \ |\sigma(t,0)| \ \leq \ K.
    \]
    \item $f$ is continuous in $x$ uniformly in $t$ and $g$ is continuous;
    \item $f$ has polynomial growth in $x$ uniformly in $t$ and $g$ has polynomial growth: there exists $q\geq0$ and $K>0$ such that, for every $t\in[0,T]$, $x\in\R^d$,
    \[
        |f(t,x)| \ + \ |g(x)| \ \leq \ K(1 + |x|^q).
    \]
    \end{enumerate}
\end{specialAssumption}

\noindent\textbf{State process.} For every $t\in[0,T]$ and $x\in\R^d$, the state process evolves according to the following stochastic differential equation:
\begin{equation}\label{SDE}
    \begin{cases}
        dX_s = b(s,X_s)ds + \sigma(s,X_s)dW_s, & s\in[t,T],\\
        X_t = x.
    \end{cases}
\end{equation}
It is well-known that under Assumption (\nameref{A1}) we have existence and uniqueness for equation \eqref{SDE}, as stated in the following proposition, whose standard proof is not reported (see for instance \cite[Chapter 9]{Baldi}).
\begin{Proposition}
    For every $t\in[0,T]$ and $x\in\R^d$, there exists a unique process $X^{t,x}\in\S^2$ (up to $\P$-indistinguishability) solution to equation \eqref{SDE} and such that $X_s^{t,x}\equiv x$, for $s\in[0,t)$. Moreover, $X^{t,x}$ satisfies the following estimate: for every $p\geq2$, there exists a constant $C_p>0$ such that
    \[
    \E\bigg[\sup_{t\leq s\leq T}|X_s^{t,x}|^p\bigg] \ \leq \ C_p(1 + |x|^p).
    \]
\end{Proposition}

\begin{Remark}\label{R:Random_initial_condition}
    It follows from Kolmogorov's continuity theorem that, up to a modification, $X^{t,x}$ is continuous with respect to $(t,x)\in[0,T]\times\R^d$ (see for instance \cite[Theorem 9.9]{Baldi}). As a consequence, given $\theta\in\Tc_{0,T}$ and $\xi\in L^2(\Fc_\theta;\R^d)$, the process $X^{\theta,\xi}$ is well-defined and belongs to $\S^2$ $($setting arbitrarily $X_s^{\theta,\xi}\equiv0$ whenever $s<\theta$$)$. Moreover, $X^{\theta,\xi}$ solves the stochastic differential equation \eqref{SDE} between $s=\theta$ and $s=T$ with initial condition $\xi$. 
    We also recall that the flow property holds for equation \eqref{SDE}: for every $\tau\in\Tc_{0,T}$, the processes $(X_s^{t,x})_{\tau\leq s\leq T}$ and $(X_s^{\tau,X_\tau^{t,x}})_{\tau\leq s\leq T}$ solve the same equation between $s=\tau$ and $s=T$, with initial condition $X_\tau^{t,x}$, therefore they are $\P$-indistinguishable.
\end{Remark}

\noindent \textbf{Optimization functional and value function.} We now introduce the gain functional $J$ and the value function $v$ of the optimal stopping problem. Given $t\in[0,T]$, $x\in \R^d$, we define the gain functional
\[
    J(t,x,\tau) \ \coloneqq \ \E\bigg[\int_{t}^{\tau} f(s,X^{t,x}_s)ds \ + \ g(X^{t,x}_\tau)\bigg],\quad\text{for all }\tau\in\Tc_{t,T},
\]
and the value function
\[
    v(t,x) \ \coloneqq \ \sup_{\tau\in\Tc_{t,T}} J(t,x,\tau).
\]
\noindent Under Assumption (\nameref{A1}), both $J$ and $v$ are well-defined. In particular, we have the following result, whose standard proof is not reported.
\begin{Proposition}\label{P:v}
    The value function $v$ is continuous and has polynomial growth in $x$ uniformly in $t$: there exists $C>0$ such that, for all $t\in[0,T]$, $x\in\R^d$,
        \[
            |v(t,x)| \ \leq \ C(1+|x|^q).
        \]
\end{Proposition}

\noindent We also introduce the function
\[
\tilde v(t,x) \ \coloneqq \ \sup_{\tau\in\tilde\Tc_{t,T}} J(t,x,\tau).
\]
This is another value function for the optimal stopping problem, characterized by a smaller set of admissible stopping times. We will prove that the functions $v$ and $\tilde v$ indeed coincide, see Corollary \ref{C:v=v_tilde}. We firstly provide a key equality for $\tilde v$ (and hence for $v$), which is stated in the following theorem (see Remark \ref{R:Random_initial_condition} for the definition of $X^{\theta,\xi}$).

\begin{Theorem}\label{T:v_tilde}
Let $\theta\in\Tc_{0,T}$ and $\xi\in L^2(\Fc_\theta;\R^d)$. Then
\[
    \tilde v(\theta,\xi) \ = \ \esssup_{\tau\in\Tc_{\theta,T}} \E\bigg[\int_{\theta}^{\tau} f(s,X^{\theta,\xi}_s)ds \ + \ g(X^{\theta,\xi}_\tau) \bigg|\Fc_\theta\bigg], \qquad \P\text{-a.s.}
\]
\end{Theorem}
\begin{proof}
    We split the proof into two main steps.
   
   \vspace{2mm}
   
   \noindent\textsc{Step} 1. We prove the inequality
    \[
        \tilde v(\theta,\xi) \ \geq \ \esssup_{\tau\in\Tc_{\theta,T}} \E\bigg[\int_{\theta}^{\tau} f(s,X^{\theta,\xi}_s)ds \ + \ g(X^{\theta,\xi}_\tau) \bigg|\Fc_\theta\bigg]
    \]
    through three steps.

    \vspace{2mm}
    
    \noindent\textsc{Step} 1.1. We assume that $\theta\equiv t\in[0,T]$ and $\xi = \sum_{i=1}^n x_i\mathbbm{1}_{A_i}$, with $x_i\in\R^d$ and $(A_i)_i\subset\Fc_t$ being a partition of $\Omega$. Then $\tilde v(t,\xi)=\sum_{i=1}^n \tilde v(t,x_i)\,\mathbbm{1}_{A_i}$. By definition of $\tilde v(t,x_i)$, for every $\tau\in\tilde\Tc_{t,T}$, we have (using also that $\tau$ is independent of $\Fc_t$)
    \begin{align*}
        \tilde v(t,\xi) \ &\geq \ \sum_{i=1}^n \E\bigg[\int_t^{\tau} f(s, X^{t,x_i}_s)ds + g(X^{t,x_i}_{\tau})\bigg]\mathbbm{1}_{A_i} \ = \ \sum_{i=1}^n \E\bigg[\int_t^{\tau} f(s, X^{t,x_i}_s)ds + g(X^{t,x_i}_{\tau})\bigg|\Fc_t\bigg]\mathbbm{1}_{A_i} \\
        &= \ \E\bigg[\sum_{i=1}^n\bigg(\int_t^{\tau} f(s, X^{t,\xi}_s)ds + g(X^{t,\xi}_{\tau})\bigg)\mathbbm{1}_{A_i}\bigg|\Fc_t\bigg] \ = \ \E\bigg[\int_t^{\tau} f(s, X^{t,\xi}_s)ds + g(X^{t,\xi}_{\tau})\bigg|\Fc_t\bigg].
    \end{align*} 
   By linearity, the above inequality is satisfied by any stopping time of the form $\tau = \sum_{i=1}^n\tau_i\mathbbm{1}_{A_i}$, where $\tau_i\in\tilde\Tc_{t,T}$ and $(A_i)_i$ is a partition of $\Omega$.
   
   \vspace{2mm}
   
   \noindent\textsc{Step} 1.2. We assume that $\theta\equiv t\in[0,T]$ and $\xi$ is as in \textsc{Step} 1.1, while $\tau\in\Tc_{t,T}$. From Theorem \ref{T:Approx} it follows that there exists a sequence $(\tau_n)_n\subset\Tc_{t,T}$ which converges $\P$-a.s. to $\tau$, where $\tau_n = \sum_{i=1}^{K_n}\tau_{n,i}\mathbbm{1}_{A_{n,i}}$, with $\tau_{n,i}\in\tilde\Tc_{t,T}$ and $A_{n,1},...,A_{n,K_n}\in\Fc_t$ being a partition of $\Omega$. By \textsc{Step} 1.1, for each $n$ we have 
    \[
        \tilde v(t,\xi) \ \geq \  \E\bigg[\int_t^{\tau_n} f(s, X^{t,\xi}_s)ds + g(X^{t,\xi}_{\tau_n})\bigg|\Fc_t\bigg].
    \]
    The claim follows letting $n\rightarrow+\infty$ and using the conditional dominated convergence theorem.

    \vspace{2mm}
    
    \noindent\textsc{Step} 1.3. We assume that $\theta\equiv t\in[0,T]$ and consider a generic $\xi\in L^2(\Fc_t;\R^d)$. By a standard result, there exists a sequence $(\xi_n)_n\subset L^2(\Fc_t;\R^d)$ such that $\xi_n$ converges to $\xi$ in $L^2$ and $\P$-a.s., moreover
    \[
        \xi_n = \sum_{i=1}^{N_n} x_{i,n}\mathbbm{1}_{A_{n,i}},
    \]
    with $x_{i,n}\in\R^d$ and $\{A_{n,1},...,A_{n,N_n}\}\subset\Fc_t$ being a partition of $\Omega$. By \textsc{Step} 1.2, for any $\tau\in\Tc_{t,T}$ it holds that
    \[
        \tilde v(t,\xi_n) \ \geq \   \E\bigg[\int_t^{\tau} f(s, X^{t,\xi_n}_s)ds + g(X^{t,\xi_n}_{\tau})\bigg|\Fc_t\bigg].
    \]
    As in the previous step, the claim follows letting $n\rightarrow+\infty$ and using the conditional dominated convergence theorem.

    \vspace{2mm}
    
    \noindent\textsc{Step} 1.4. We consider a generic $\xi\in L^2(\Fc_\theta;\R^d)$ and assume that $\theta=\sum_{i=1}^n t_i\mathbbm{1}_{B_i}$, with $t_i\in[0,T]$, $B_i\in\Fc_{t_i}$, $(B_i)_i$ being a partition of $\Omega$. Given $\tau\in\Tc_{\theta,T}$, let $\tau_i:=\tau\mathbbm{1}_{B_i}+T\mathbbm{1}_{B_i^c}$; notice that $\tau_i\in\Tc_{t_i,T}$. Then, from \textsc{Step} 1.3 we get
    \begin{align*}
        \tilde v(\theta,\xi) \ = \ \sum_{i=1}^n \tilde v(t_i,\xi)\,\mathbbm 1_{B_i} \ &\geq \ \sum_{i=1}^n \E\bigg[\int_{t_i}^{\tau_i} f(s, X^{t_i,\xi}_s)ds + g(X^{t_i,\xi}_{\tau_i})\bigg|\Fc_{t_i}\bigg]\mathbbm 1_{B_i} \\
        &= \ \sum_{i=1}^n \E\bigg[\bigg(\int_{t_i}^{\tau_i} f(s, X^{t_i,\xi}_s)ds + g(X^{t_i,\xi}_{\tau_i})\bigg)\mathbbm 1_{B_i}\bigg|\Fc_{t_i}\bigg] \\
        &= \ \sum_{i=1}^n \E\bigg[\bigg(\int_\theta^\tau f(s, X^{\theta,\xi}_s)ds + g(X^{\theta,\xi}_\tau)\bigg)\mathbbm 1_{B_i}\bigg|\Fc_{t_i}\bigg].       
    \end{align*}
    Now, notice that given $Y\in L^1$ it holds that $\E[Y\mathbbm 1_{B_i}|\Fc_{t_i}]=\E[Y\mathbbm 1_{B_i}|\Fc_\theta]$, then
    \[
        \tilde v(\theta,\xi) \ \geq \ \sum_{i=1}^n \E\bigg[\bigg(\int_\theta^\tau f(s, X^{\theta,\xi}_s)ds + g(X^{\theta,\xi}_\tau)\bigg)\mathbbm 1_{B_i}\bigg|\Fc_\theta\bigg] \ = \ \E\bigg[\int_\theta^\tau f(s, X^{\theta,\xi}_s)ds + g(X^{\theta,\xi}_\tau)\bigg|\Fc_\theta\bigg],
    \]
    which concludes the proof of \textsc{Step} 1.4.

    \vspace{2mm}
    
    \noindent\textsc{Step} 1.5. We consider a generic $\theta\in\Tc_{0,T}$ and a generic $\xi\in L^2(\Fc_\theta;\R^d)$. It is well-known that there exists a sequence $(\theta_n)_n\subset\Tc_{0,T}$ such that $\theta_n\downarrow\theta$ a.s., with each $\theta_n$ as in \textsc{Step} 1.4 (see for instance \cite[Lemma 3.3]{Baldi}). By \textsc{Step} 1.4 we have, for each $n$,
    \[
    \tilde v(\theta_n,\xi) \ \geq \ \E\bigg[\int_{\theta_n}^\tau f(s, X^{\theta_n,\xi}_s)ds + g(X^{\theta_n,\xi}_\tau)\bigg|\Fc_{\theta_n}\bigg].
    \]
    Let $Y:=\int_\theta^\tau f(s, X^{\theta,\xi}_s)ds + g(X^{\theta,\xi}_\tau)$ and $Y_n:=\int_{\theta_n}^\tau f(s, X^{\theta_n,\xi}_s)ds + g(X^{\theta_n,\xi}_\tau)$. Then, the claim follows from the fact that $\E[Y_n|\Fc_{\theta_n}]$ converges in $L^1$ to $\E[Y|\Fc_\theta]$. To prove this, we notice that
    \begin{align*}
    &\E\big[\big|\E[Y_n|\Fc_{\theta_n}]-\E[Y|\Fc_\theta]\big|\big] \ = \ \E\big[\big|\E[Y_n|\Fc_{\theta_n}]-\E[Y|\Fc_{\theta_n}]+\E[Y|\Fc_{\theta_n}]-\E[Y|\Fc_\theta]\big|\big] \\
    &\leq \ \E\big[\big|\E[Y_n-Y|\Fc_{\theta_n}]\big|\big] + \E\big[\big|\E[Y|\Fc_{\theta_n}]-\E[Y|\Fc_\theta]\big|\big] \leq \ \E\big[|Y_n-Y|\big] + \E\big[\big|\E[Y|\Fc_{\theta_n}]-\E[Y|\Fc_\theta]\big|\big],
    \end{align*}
    where $\E\big[|Y_n-Y|\big]$ goes to $0$ because $Y_n$ converges in $L^1$ to $Y$, and $\E\big[\big|\E[Y|\Fc_{\theta_n}]-\E[Y|\Fc_\theta]\big|\big]$ vanishes due to the backward martingale convergence theorem (see for instance \cite[Theorem 27.4]{JacodProtter}) together with the fact that $\Fc_\theta=\cap_n\Fc_{\theta_n}$.
    
    \vspace{2mm}
    
    \noindent\textsc{Step} 2. Our aim is to prove the opposite inequality
    \[
        \tilde v(\theta,\xi) \ \leq \ \esssup_{\tau\in\Tc_{\theta,T}} \E\bigg[\int_{\theta}^{\tau} f(s,X^{\theta,\xi}_s)ds \ + \ g(X^{\theta,\xi}_\tau) \bigg|\Fc_\theta\bigg].
    \]
    The proof consists in three steps, as in the case of the other inequality. We do not report the details, but we just notice that in the first step, when $\xi = \sum_{i=1}^n x_i\mathbbm{1}_{A_i}$, using the definition of $\tilde v(t,x_i)$, for every $\varepsilon>0$ and any $i$, there exists $\tau_{\varepsilon,i}\in\tilde\Tc_{t,T}$ such that
    \begin{align*}
        &\tilde v(t,\xi) \ = \ \sum_{i=1}^n \tilde v(t,x_i)\mathbbm{1}_{A_i} \ \leq \ \sum_{i=1}^n \E\bigg[\int_t^{\tau_{\varepsilon,i}} f(s, X^{t,x_i}_s)ds + g(X^{t,x_i}_{\tau_{\varepsilon,i}})\bigg]\mathbbm{1}_{A_i} + \varepsilon \\
        &= \ \E\bigg[\sum_{i=1}^n\bigg(\int_t^{\tau_{\varepsilon,i}} \!\! f(s, X^{t,x_i}_s)ds + g(X^{t,x_i}_{\tau_{\varepsilon,i}})\bigg)\mathbbm{1}_{A_i}\bigg|\Fc_t\bigg] + \varepsilon \ = \ \E\bigg[\int_t^{\tau_\varepsilon} f(s, X^{t,\xi}_s)ds + g(X^{t,\xi}_{\tau_\varepsilon})\bigg|\Fc_t\bigg] + \varepsilon,
    \end{align*}
    where $\tau_\varepsilon = \sum_{i=1}^n \tau_{\varepsilon,i}\mathbbm{1}_{A_i}$. The rest of the proof of \textsc{Step 2} proceeds similarly to \textsc{Step 1}.
\end{proof}
\noindent By Theorem \ref{T:v_tilde} it follows that $v$ and $\tilde v$ coincide, as stated in the following corollary.
\begin{Corollary}\label{C:v=v_tilde}
    Let $t\in[0,T]$ and $x\in\R^d$. Then 
    \[
        \tilde v(t,x) \ = \ \sup_{\tau\in\Tc_{t,T}}\E\bigg[\int_t^{\tau} f(s, X^{t,x}_s)ds + g(X^{t,x}_{\tau})\bigg] \ = \ v(t,x).
    \]
\end{Corollary}
\begin{proof}
The inequality $\tilde v\leq v$ follows immediately from the inclusion $\Tilde{\Tc}_{t,T}\subset\Tc_{t,T}$. Let us now prove the reverse inequality. An application of Theorem \ref{T:v_tilde} with $\theta\equiv t$ and $\xi\equiv x$ yields 
    \[
        \tilde v(t,x) \ = \ \esssup_{\tau\in\Tc_{t,T}}\E\bigg[\int_t^{\tau} f(s, X^{t,x}_s)ds + g(X^{t,x}_{\tau})\bigg|\Fc_t\bigg].
    \]
    Then, for every $\tau\in\Tc_{t,T}$, it holds that
    \[
        \tilde v(t,x) \ \geq \ \E\bigg[\int_t^{\tau} f(s, X^{t,x}_s)ds + g(X^{t,x}_{\tau})\bigg|\Fc_t\bigg].
    \]
    Recalling that $\Tilde{v}$ is deterministic, we get
    \[
        \tilde v(t,x) \ = \ \E[\Tilde{v}(t,x)] \ \geq \ \E\bigg[\int_t^{\tau} f(s, X^{t,x}_s)ds + g(X^{t,x}_{\tau})\bigg].
    \]
    From the arbitrariness of $\tau$, the claim follows.
\end{proof}

\noindent\textbf{Optimal stopping time.} Let $t\in[0,T]$ and $x\in\R^d$. We consider the stopping time $\hat{\tau}^{t,x}\in\tilde\Tc_{t,T}$ given by
\begin{equation}\label{OptStoppTime}
    \hat{\tau}^{t,x} \ \coloneq \ \inf\big\{s\in[t,T] \ \big| \ v(s,X^{t,x}_s) = g(X^{t,x}_s)\big\}.
\end{equation}

\begin{Theorem}\label{T:Optimal}
The stopping time $\hat{\tau}^{t,x}$ is the smallest optimal stopping time for the problem starting at time $t$ from state $x$, namely:
\begin{enumerate}[\upshape i)]
\item $v(t,x)$ $=$ $J\big(t,x,\hat{\tau}^{t,x}\big)$;
\item $\hat\tau^{t,x}\leq\tau^*$, $\P$-a.s., for any other optimal stopping time $\tau^*\in\Tc_{t,T}$.
\end{enumerate}
\end{Theorem}
\begin{proof}
Fix $\theta\in\Tc_{t,T}$. Our aim is to prove the following:
\begin{enumerate}[1)]
\item if $\P(\theta<\hat\tau^{t,x})>0$ then $J(t,x,\theta)<v(t,x)$, namely $\theta$ can not be optimal;
\item if $\hat\tau^{t,x}\leq\theta$, $\P$-a.s., then $J(t,x,\hat\tau^{t,x})\geq J(t,x,\theta)$.
\end{enumerate}
It is clear that the claim follows directly from these two properties.

\vspace{2mm}

\noindent\emph{Proof of  1)}. Let
\[
\delta_\theta \ := \ \E\big[\big(v(\theta,X_\theta^{t,x})-g(X_\theta^{t,x})\big)\mathbbm{1}_{\{\theta<\hat\tau^{t,x}\}}\big].
\]
By definition of $v$ we have that $v\geq g$, so that $\delta_\theta\geq0$. Moreover, $\delta_\theta>0$ if and only if $\P(\theta<\hat\tau^{t,x})>0$. Now, we have
\begin{align*}
&J(t,x,\theta) \ = \ \E\bigg[\int_{t}^{\theta} f(s,X^{t,x}_s)ds + g(X^{t,x}_{\theta})\bigg] \ = \ \E\bigg[\bigg(\int_{t}^{\theta} f(s,X^{t,x}_s)ds + g(X^{t,x}_{\theta})\bigg)\mathbbm{1}_{\{\theta<\hat\tau^{t,x}\}}\bigg] + \\
&+ \E\bigg[\bigg(\int_{t}^{\theta} f(s,X^{t,x}_s)ds + g(X^{t,x}_{\theta})\bigg)\mathbbm{1}_{\{\theta=\hat\tau^{t,x}\}}\bigg] + \E\bigg[\bigg(\int_{t}^{\theta} f(s,X^{t,x}_s)ds + g(X^{t,x}_{\theta})\bigg)\mathbbm{1}_{\{\theta>\hat\tau^{t,x}\}}\bigg] \\
&= \ \E\bigg[\bigg(\int_{t}^{\theta} \!\!\! f(s,X^{t,x}_s)ds + v(\theta,X^{t,x}_{\theta})\!\bigg)\mathbbm{1}_{\{\theta\leq\hat\tau^{t,x}\}}\bigg] - \delta_\theta + \E\bigg[\bigg(\int_{t}^{\theta} \!\!\! f(s,X^{t,x}_s)ds + g(X^{t,x}_{\theta})\!\bigg)\mathbbm{1}_{\{\theta>\hat\tau^{t,x}\}}\bigg].
\end{align*}
By Theorem \ref{T:v_tilde} and the flow property (see Remark \ref{R:Random_initial_condition}), we obtain that there exists $\tau_\varepsilon\in\Tc_{\theta,T}$ such that
\begin{align}
J(t,x,\theta) \ &\leq \ \E\bigg[\bigg(\int_{t}^{\theta} f(s,X^{t,x}_s)ds + \E\bigg[\int_{\theta}^{\tau_\varepsilon} f\big(s,X^{\theta,X_\theta^{t,x}}_s\big)ds + g\big(X^{\theta,X_\theta^{t,x}}_{\tau_\varepsilon}\big) \bigg|\Fc_\theta\bigg]\bigg)\mathbbm{1}_{\{\theta\leq\hat\tau^{t,x}\}}\bigg] + \notag \\
&\quad \ + \varepsilon\,\P(\theta\leq\hat\tau^{t,x}) - \delta_\theta + \E\bigg[\bigg(\int_{t}^{\theta} f(s,X^{t,x}_s)ds + g(X^{t,x}_{\theta})\bigg)\mathbbm{1}_{\{\theta>\hat\tau^{t,x}\}}\bigg] \notag \\
&= \ \E\bigg[\bigg(\int_{t}^{\tau_\varepsilon} f(s,X^{t,x}_s)ds + g(X^{t,x}_{\tau_\varepsilon})\bigg)\mathbbm{1}_{\{\theta\leq\hat\tau^{t,x}\}}\bigg] + \varepsilon\,\P(\theta\leq\hat\tau^{t,x}) - \delta_\theta \label{Proof_OptStoppStep1} \\
&\quad \ + \E\bigg[\bigg(\int_{t}^{\theta} f(s,X^{t,x}_s)ds + g(X^{t,x}_{\theta})\bigg)\mathbbm{1}_{\{\theta>\hat\tau^{t,x}\}}\bigg]. \notag
\end{align}
Let $\check\tau:=\tau_\varepsilon\mathbbm{1}_{\{\theta\leq\hat\tau^{t,x}\}}+\theta\mathbbm{1}_{\{\theta>\hat\tau^{t,x}\}}$. Notice that $\check\tau\in\Tc_{\theta,T}$. Indeed, by definition $\theta\leq\check\tau\leq T$, $\P$-a.s.. Moreover, for every $s\in[0,T]$,
\[
\{\check\tau\leq s\} \ = \ \big(\{\tau_\varepsilon\leq s\}\cap\{\theta\leq\hat\tau^{t,x}\}\big)\cup\big(\{\theta\leq s\}\cap\{\theta>\hat\tau^{t,x}\}\big).
\]
Since $\tau_\varepsilon\geq\theta$, we have that $\{\tau_\varepsilon\leq s\}=\{\tau_\varepsilon\leq s\}\cap\{\theta\leq s\}$. Furthermore, both sets $\{\theta\leq\hat\tau^{t,x}\}$ and $\{\theta>\hat\tau^{t,x}\}$ belong to $\Fc_\theta$. From these facts it follows that $\{\check\tau\leq s\}\in\Fc_s$. Now, using $\check\tau$ in \eqref{Proof_OptStoppStep1} we obtain (bounding from above $\P(\theta\leq\hat\tau^{t,x})$ by $1$)
\[
J(t,x,\theta) \ \leq \ \E\bigg[\int_{t}^{\check\tau} f(s,X^{t,x}_s)ds + g(X^{t,x}_{\check\tau})\bigg] + \varepsilon - \delta_\theta \ = \ J(t,x,\check\tau) + \varepsilon - \delta_\theta \ \leq \ v(t,x) + \varepsilon - \delta_\theta.
\]
From the arbitrariness of $\varepsilon$, we deduce that $J(t,x,\theta)\leq v(t,x)-\delta_\theta$, from which the claim follows.

\vspace{2mm}

\noindent\emph{Proof of 2)}. We have
\[
J(t,x,\hat\tau^{t,x}) \ = \ \E\bigg[\int_{t}^{\hat\tau^{t,x}} f(s,X^{t,x}_s)ds + g(X^{t,x}_{\hat\tau^{t,x}})\bigg] \ = \ \E\bigg[\int_{t}^{\hat\tau^{t,x}} f(s,X^{t,x}_s)ds + v(\hat\tau^{t,x},X^{t,x}_{\hat\tau^{t,x}})\bigg],
\]
where the second inequality follows from the definition \eqref{OptStoppTime} of $\hat\tau^{t,x}$. By Theorem \ref{T:v_tilde} and the flow property (see Remark \ref{R:Random_initial_condition}), for every $\theta\in\Tc_{\hat\tau^{t,x},T}$ we get
\begin{align*}
J(t,x,\hat\tau^{t,x}) \ &\geq \ \E\bigg[\int_{t}^{\hat\tau^{t,x}} f(s,X^{t,x}_s)ds + \E\bigg[\int_{\hat\tau^{t,x}}^{\theta} f\big(s,X^{\hat\tau^{t,x},X_{\hat\tau^{t,x}}^{t,x}}_s\big)ds + g\big(X^{\hat\tau^{t,x},X_{\hat\tau^{t,x}}^{t,x}}_{\theta}\big) \bigg|\Fc_{\hat\tau^{t,x}}\bigg]\bigg] \\
&= \ \E\bigg[\int_t^\theta f(s,X_s^{t,x})ds + g(X_\theta^{t,x})\bigg] \ = \ J(t,x,\theta),
\end{align*}
which shows the validity of item 2).
\end{proof}

\begin{Remark}[Snell envelope]\label{R:Snell}
Theorem \ref{T:Optimal} can also be proven using the martingale approach to optimal stopping. Specifically, given $(t,x)\in[0,T]\times\R^d$, let us introduce the gain process
\[
    G^{t,x}_s \ \coloneqq \ \int_{t}^{s}f\big(r,X^{t,x}_r\big)dr \ + \ g\big(X^{t,x}_s\big),\qquad s\in[t,T].
\]
From the assumptions on the coefficients $b$, $\sigma$, $f$, $g$ it follows that $\E[\sup_{t\leq s\leq T}|G^{t,x}_s|]<+\infty$. We can then define the process
\begin{equation}\label{Snell}
    \Sc^{t,x}_s \ \coloneqq \ \esssup_{\tau\in\Tc_{s,T}}\E\big[G^{t,x}_\tau\big|\Fc_s\big],\qquad s\in[t,T].
\end{equation}
By \cite[Theorem 2.2]{PeskirShiryaev}, we know that $\Sc^{t,x}$ is the Snell envelope of $G^{t,x}$, and the smallest optimal stopping time is given by
\[
    \hat{\theta}^{t,x} \ \coloneqq \ \inf\big\{s\in[t,T]\ \big| \ \Sc^{t,x}_s = G^{t,x}_s\big\}.
\]
Thus, Theorem \ref{T:Optimal} holds if we prove that $\hat{\theta}^{t,x}$ coincides with $\hat\tau^{t,x}$ in \eqref{OptStoppTime}. This, in turn, follows from the fact that the Snell envelope $\Sc^{t,x}$ is equal to the process
\[
    \bigg(v(s,X^{t,x}_s) + \int_{t}^{s} f\big(r,X^{t,x}_r\big)dr\bigg)_{s\in[t,T]}.
\]
To see this, notice that from \eqref{Snell} we obtain
\begin{align*}
    \Sc^{t,x}_s \ &= \ \esssup_{\tau\in\Tc_{s,T}}\E\bigg[\int_{t}^{s}f(r,X^{t,x}_r)dr + \int_{s}^{\tau}f(r,X^{t,x}_r)dr + g(X^{t,x}_\tau)\bigg|\Fc_s\bigg] \\
    &= \ \int_{t}^{s}f(r,X^{t,x}_r)dr + \esssup_{\tau\in\Tc_{s,T}}\E\bigg[\int_{s}^{\tau}f(r,X^{t,x}_r)dr + g(X^{t,x}_\tau)\bigg|\Fc_s\bigg].
\end{align*}
The claim follows if we prove that
\begin{equation}\label{Vtilde = esssup_Claim}
    \esssup_{\tau\in\Tc_{s,T}}\E\bigg[\int_{s}^{\tau}f(r,X^{t,x}_r)dr \ + \ g(X^{t,x}_\tau)\bigg|\Fc_s\bigg] \ = \ v(s,X^{t,x}_s).
\end{equation}
Recalling Theorem \ref{T:v_tilde}, we have
\[
    v(s,X^{t,x}_s) \ = \ \esssup_{\tau\in\Tc_{s,T}}\E\bigg[\int_{s}^{\tau}f\big(r,X^{s,X^{t,x}_s}_r\big)dr + g\big(X^{s,X^{t,x}_s}_\tau\big)\bigg|\Fc_s\bigg].
\]
From the flow property (see Remark \ref{R:Random_initial_condition}), we have that $(X_r^{s,X_s^{t,x}})_{r\in[s,T]}$ and $(X_r^{s,x})_{r\in[s,T]}$ are $\P$-indistinguishable, therefore \eqref{Vtilde = esssup_Claim} follows. This proves the characterization of the Snell envelope $\Sc^{t,x}$, and, consequently, the fact that the stopping times $\hat\theta^{t,x}$ and $\hat\tau^{t,x}$ coincide.
\end{Remark}

\noindent\textbf{Dynamic programming principle.} We now exploit Theorem \ref{T:v_tilde} and Corollary \ref {C:v=v_tilde} in order to prove rigorously the dynamic programming principle for $v$.
\begin{Theorem}\label{T:DPP}
    Let $t\in[0,T]$ and $x\in\R^d$. Then, the following dynamic programming principle for $v$ holds: for every $\tau'\in\Tc_{t,T}$,
    \begin{equation}
        v(t,x) \ = \ \sup_{\tau\in\Tc_{t,T}} \E\bigg[\int_{t}^{\tau\land\tau'} f(s,X^{t,x}_s)ds + g(X^{t,x}_\tau)\mathbbm{1}_{\{\tau<\tau'\}} + v(\tau',X^{t,x}_{\tau'})\mathbbm{1}_{\{\tau\geq\tau'\}}\bigg]. \label{DPP MKV}
    \end{equation}
\end{Theorem}
\begin{proof}
    Set
    \begin{equation}\label{LambdaProofStep1}
        \Lambda(t,x) \coloneqq \sup_{\tau\in \Tc_{t,T}} \E\bigg[\int_{t}^{\tau\land \tau'} f(s,X^{t,x}_s)ds + g(X^{t,x}_\tau)\mathbbm{1}_{\{\tau<\tau'\}} + v(\tau',X^{t,x}_{\tau'})\mathbbm{1}_{\{\tau\geq\tau'\}}\bigg]. 
    \end{equation}
    \textsc{Step} 1: $v(t,x)\leq\Lambda(t,x)$. For every $\tau\in\Tc_{t,T}$, we have
    \[
        \Lambda(t,x) \ \geq \ \E\bigg[\int_{t}^{\tau\land\tau'} f(s,X^{t,x}_s)ds + g(X^{t,x}_\tau)\mathbbm{1}_{\{\tau<\tau'\}} + v(\tau',X^{t,x}_{\tau'})\mathbbm{1}_{\{\tau\geq\tau'\}}\bigg].
    \]
    By Theorem \ref{T:v_tilde} and Corollary \ref{C:v=v_tilde}, we have (noting that $\tau\vee\tau'\in\Tc_{\tau',T}$)
    \begin{align*}
        v(\tau',X^{t,x}_{\tau'}) \ &\geq\ \E\bigg[\int_{\tau'}^{\tau\vee \tau'} f\big(s,X^{\tau',X^{t,x}_{\tau'}}_s\big)ds\ + \ g\big(X^{\tau',X^{t,x}_{\tau'}}_{\tau\vee \tau'}\big)\bigg|\Fc_{\tau'}\bigg] \\
        &= \ \E\bigg[\int_{\tau'}^{\tau\vee \tau'} f(s,X^{t,x}_s)ds + g(X^{t,x}_{\tau\vee \tau'})\bigg|\Fc_{\tau'}\bigg],
    \end{align*}
    where the last equality follows from the flow property (see Remark \ref{R:Random_initial_condition}). Plugging the above inequality into \eqref{LambdaProofStep1} and using that $\{\tau\geq\tau'\}\in\Fc_{\tau'}$, we deduce that $\Lambda(t,x)\geq J(t,x,\tau)$. From the arbitrariness of $\tau\in\Tc_{t,T}$, we get $\Lambda(t,x)\geq v(t,x)$.

    \vspace{2mm}
    
    \noindent\textsc{Step} 2: $v(t,x)\geq\Lambda(t,x)$. For every $\varepsilon>0$ there exists $\tau_\varepsilon\in\Tc_{t,T}$ such that
    \[
        \Lambda(t,x) \ \leq \ \E\bigg[\int_{t}^{\tau_\varepsilon\land \tau'} f(s,X^{t,x}_s)ds + g(X^{t,x}_{\tau_\varepsilon})\mathbbm{1}_{\{\tau_\varepsilon<\tau'\}} + v(\tau',X^{t,x}_{\tau'})\mathbbm{1}_{\{\tau_\varepsilon\geq \tau'\}}\bigg] + \varepsilon.
    \]
    By Theorem \ref{T:v_tilde} and Corollary \ref{C:v=v_tilde}, there exists $\tilde\tau_\varepsilon\in\Tc_{\tau',T}$ such that
    \begin{align*}
        v(\tau',X^{t,x}_{\tau'}) \ &\leq \ \E\bigg[\int_{\tau'}^{\tilde\tau_\varepsilon} f(s,X^{\tau',X^{t,x}_{\tau'}}_s)ds + g(X^{\tau',X^{t,x}_{\tau'}}_{\tilde\tau_\varepsilon})\bigg|\Fc_{\tau'}\bigg] + \varepsilon \\
        &= \ \E\bigg[\int_{\tau'}^{\tilde\tau_\varepsilon} f(s,X^{t,x}_s)ds + g(X^{t,x}_{\tilde\tau_\varepsilon})\bigg|\Fc_{\tau'}\bigg] + \varepsilon,
    \end{align*}
    where the last equality follows from the flow property (see Remark \ref{R:Random_initial_condition}). This implies that
    \begin{align*}
        &\Lambda(t,x) \ \leq \\
        &\leq \ \E\bigg[\int_{t}^{\tau_\varepsilon\land \tau'} f(s,X^{t,x}_s)ds + g(X^{t,x}_{\tau_\varepsilon})\mathbbm{1}_{\{\tau_\varepsilon<\tau'\}} + \bigg(\int_{\tau'}^{\tilde\tau_\varepsilon} f(s,X^{t,x}_s)ds + g(X^{t,x}_{\tilde\tau_\varepsilon})\bigg)\mathbbm{1}_{\{\tau_\varepsilon\geq \tau'\}}\bigg] + 2\varepsilon \\
        &= \ \E\bigg[\int_{t}^{\hat\tau_\varepsilon\land \tau'} f(s,X^{t,x}_s)ds + g(X^{t,x}_{\hat\tau_\varepsilon})\mathbbm{1}_{\{\hat\tau_\varepsilon<\tau'\}} + \mathbbm{1}_{\{\hat\tau_\varepsilon \geq \tau'\}}\bigg(\int_{\tau'}^{\hat\tau_\varepsilon} f(s,X^{t,x}_s)ds + g(X^{t,x}_{\hat\tau_\varepsilon})\bigg)\bigg] + 2\varepsilon,
    \end{align*}
   where $\hat\tau_\varepsilon \coloneqq \tau_\varepsilon\mathbbm{1}_{\{\tau_\varepsilon<\tau'\}}+\tilde\tau_\varepsilon\mathbbm{1}_{\{\tau_\varepsilon\geq \tau'\}}$. This shows that $\Lambda(t,x)\leq J(t,x,\hat\tau_\varepsilon)+2\varepsilon\leq v(t,x)+2\varepsilon$. The claim follows from the arbitrariness of $\varepsilon$.
\end{proof}

\noindent\textbf{Hamilton-Jacobi-Bellman equation.} Thanks to the dynamic programming principle (Theorem \ref{T:DPP}) we now show that the value function $v$ is a viscosity solution of the following Hamilton-Jacobi-Bellman equation:
\begin{equation}\label{HJB}
    \begin{cases}
    \vspace{2mm}\max\big\{\partial_t u + \Lc_t u + f, g - u\big\}=0, &\qquad \text{on }[0,T)\times\R^d, \\
    u(T,x) = g(x), &\qquad x\in\R^d,
\end{cases}
\end{equation}
where
\[
    \Lc_t u(t,x) \ = \ \big\langle b(t,x),\partial_x u(t,x)\big\rangle + \frac{1}{2}\text{tr}\big(\sigma\sigma^T (t,x)\partial^2_x u(t,x)\big).
\]
We now recall the usual definition of viscosity solution for equation \eqref{HJB}.
\begin{Definition}
    A continuous map $u\colon[0,T]\times\R^d\longrightarrow\R$ is a viscosity subsolution (resp. supersolution) to \eqref{HJB} if:
    \begin{itemize}
        \item $u(T,x)\leq \,(\text{resp.}\geq)\ g(x)$, $x\in\R^d$;
        \item given $(t,x)\in[0,T)\times\R^d$ and $\varphi\in\Cc^{1,2}([0,T]\times\R^d)$ such that $u-\varphi$ has a local maximum (resp. minimum) at $(t,x)$ with value $0$, it holds that 
        \[
            \max\big\{\partial_t \varphi+ \Lc_t \varphi + f, g - \varphi\big\}\geq \,(\text{resp.} \leq)\ 0.
        \]
    \end{itemize}
    Moreover, $u$ is a viscosity solution to \eqref{HJB} if it is both a viscosity subsolution and a viscosity supersolution.
\end{Definition}

\begin{Theorem}\label{T: V tilde solution HJB}
    In addition to \textnormal{Assumption (\nameref{A1})} we suppose that $b,\sigma, f$ are continuous on $[0,T]\times\R^d$. Then, the value function $v$ is a viscosity solution of equation \eqref{HJB}.
\end{Theorem}

\begin{proof} 
    By Proposition \ref{P:v} we know that $v$ is a continuous function. Furthermore, by definition we have that $v(T,x)=g(x)$, $\forall\,x\in\R^d$. Let us then prove the viscosity properties of $v$ on $[0,T)\times\R^d$.

    \vspace{2mm}
    
    \noindent\emph{Supersolution property.} Consider $(t,x)\in[0,T)\times\R^d$ and $\varphi\in\Cc^{1,2}([0,T]\times\R^d)$, with $v-\varphi$ attaining a local minimum at $(t,x)$ with value $0$. Let
    \[
    \tau_h' \ := \ \inf\big\{s\in[t,T] \, \big| \, v(s,X_s^{t,x})<\varphi(s,X_s^{t,x})\big\}\wedge(t+h),
    \]
    with $\inf\emptyset:=T$ and $h\in(0,T-t)$. Using the dynamic programming principle \eqref{DPP MKV}, we get
    \[
        v(t,x) \ \geq \ \E\bigg[\int_{t}^{\tau\land\tau'_h} f(s,X^{t,x}_s)ds + g(X^{t,x}_\tau)\mathbbm{1}_{\{\tau<\tau'_h\}} + v(\tau'_h,X^{t,x}_{\tau'_h})\mathbbm{1}_{\{\tau\geq\tau'_h\}}\bigg],\quad\text{for all }\tau\in\Tc_{t,T}.
    \]
    By $v(t,x)=\varphi(t,x)$ and $v(\tau'_h,X_{\tau'_h}^{t,x})\geq\varphi(\tau'_h,X_{\tau'_h}^{t,x})$, we obtain
    \[
        \varphi(t,x) \ \geq \ \E\bigg[\int_{t}^{\tau\land\tau'_h} f(s,X^{t,x}_s)ds \ + \ g(X^{t,x}_\tau)\mathbbm{1}_{\{\tau<\tau'_h\}} + \varphi(\tau'_h,X^{t,x}_{\tau'_h})\mathbbm{1}_{\{\tau\geq\tau'_h\}}\bigg],\quad\text{for all }\tau\in\Tc_{t,T}.
    \]
    An application of It\^o's formula to $\varphi(\tau'_h,X^{t,x}_{\tau'_h})$, yields, for all $\tau\in\Tc_{t,T}$,
    \begin{align*}
        \varphi(t,x) \ &\geq \ \E\bigg[\int_{t}^{\tau\land\tau'_h} f(s,X^{t,x}_s)ds + g(X^{t,x}_\tau)\mathbbm{1}_{\{\tau<\tau'_h\}} \\
        &\quad \ + \bigg(\varphi(t,x) + \int_t^{\tau'_h} \big(\partial_t v + \Lc_t\varphi\big)(s,X^{t,x}_s)ds\bigg)\mathbbm{1}_{\{\tau\geq\tau'_h\}}\bigg].
    \end{align*}
    Now, choosing $\tau=t$ we get
    \begin{equation}\label{ProofStep1.1}
        \varphi(t,x) \ \geq \ \E[g(X^{t,x}_t)] \ = \ g(x).
    \end{equation}
    Choosing instead $\tau=\tau'_h$, we find
    \[
        0 \ \geq \ \E\bigg[\int_t^{\tau'_h}\big(\partial_t\varphi+\Lc_t\varphi + f\big)(s,X^{t,x}_s)ds\bigg].
    \]
    Multiplying by $h^{-1}$ and sending $h\to 0^+$, we obtain
    \begin{equation}\label{ProofStep1.2}
        0\ \geq \ \big(\partial_t\varphi+\Lc_t\varphi + f\big)(t,x).
    \end{equation}
    In conclusion, by \eqref{ProofStep1.1} and \eqref{ProofStep1.2}, we get
   \[
        \max\big\{(\partial_t \varphi + \Lc_t \varphi + f)(t,x), g(x) - \varphi(t,x)\big\} \ \leq \ 0.
    \]
    \emph{Subsolution property.} Consider $(t,x)\in[0,T)\times\R^d$ and $\varphi\in\Cc^{1,2}([0,T]\times\R^d)$, with $v-\varphi$ attaining a local maximum at $(t,x)$ with value $0$. Our aim is to show that
    \[
        \max\big\{(\partial_t \varphi + \Lc_t \varphi + f)(t,x), g(x) - \varphi(t,x)\big\} \ \geq \ 0.
    \]
    If $g(x) - \varphi(t,x)\geq 0$, then there is nothing to prove. Let us then suppose that $g(x) - \varphi(t,x)< 0$. Set
    \[
    \tau_h' \ := \ \inf\big\{s\in[t,T] \, \big| \, v(s,X_s^{t,x})>\varphi(s,X_s^{t,x}) \ \wedge \ g(X_s^{t,x})>\varphi(t,x)\big\}\wedge(t+h),
    \]
    with $\inf\emptyset:=T$ and $h\in(0,T-t)$. Using the dynamic programming principle \eqref{DPP MKV}, we deduce the existence of $\tau_h\in\Tc_{t,T}$ such that
    \[
        v(t,x) \ \leq \ \E\bigg[\int_{t}^{\tau_h\land \tau'_h} f(s,X^{t,x}_s)ds  +  g(X^{t,x}_{\tau_h})\mathbbm{1}_{\{{\tau_h}<\tau'_h\}} +  v(\tau'_h,X^{t,x}_{\tau'_h})\mathbbm{1}_{\{{\tau_h}\geq \tau'_h\}}\bigg] + h^2.
    \]
    By $v(t,x)=\varphi(t,x)$ and $v(\tau'_h,X^{t,x}_{\tau'_h})\leq\varphi(\tau'_h,X^{t,x}_{\tau'_h})$, we obtain
    \begin{align*}
        \varphi(t,x) \ \leq \ \E\bigg[\int_{t}^{\tau_h\land \tau'_h} f(s,X^{t,x}_s)ds  +  g(X^{t,x}_{\tau_h})\mathbbm{1}_{\{{\tau_h}<\tau'_h\}} +  \varphi(\tau'_h,X^{t,x}_{\tau'_h})\mathbbm{1}_{\{{\tau_h}\geq \tau'_h\}}\bigg] + h^2.
    \end{align*}
    An application of It\^o's formula to $\varphi(\tau'_h,X^{t,x}_{\tau'_h})$, yields, for all $\tau\in\Tc_{t,T}$,
   \begin{align*}
        \varphi(t,x) \ &\leq \ \E\bigg[\int_{t}^{\tau_h\land\tau'_h} f(s,X^{t,x}_s)ds + g(X^{t,x}_{\tau_h})\mathbbm{1}_{\{\tau_h<\tau'_h\}} \\
        &\quad \ + \bigg(\varphi(t,x) + \int_t^{\tau'_h} \big(\partial_t v + \Lc_t\varphi\big)(s,X^{t,x}_s)ds\bigg)\mathbbm{1}_{\{\tau_h\geq\tau'_h\}}\bigg] + h^2.
    \end{align*}
    This can be written as
    \begin{align*}
        0 \ &\leq \ \E\bigg[\int_{t}^{\tau_h\land\tau'_h} f(s,X^{t,x}_s)ds + \mathbbm{1}_{\{{\tau_h}<\tau'_h\}}\big(g(X^{t,x}_{\tau_h})-\varphi(t,x)\big)\notag\\
        &\quad \  + \mathbbm{1}_{\{{\tau_h}\geq \tau'_h\}}\int_t^{\tau'_h}\big(\partial_t\varphi + \Lc_t\varphi\big)(s,X^{t,x}_s)ds\bigg] + h^2\notag\\
        &\leq \ \E\bigg[\mathbbm{1}_{\{{\tau_h}<\tau'_h\}}\int_{t}^{\tau_h} f(s,X^{t,x}_s)ds + \mathbbm{1}_{\{{\tau_h}\geq \tau'_h\}}\int_t^{\tau'_h}\big(\partial_t\varphi + \Lc_t\varphi + f\big)(s,X^{t,x}_s)ds\bigg] + h^2,
    \end{align*}
    where the latter equality follows from the fact that $g(X^{t,x}_{\tau_h})\leq\varphi(t,x)$ on $\{\tau_h<\tau'_h\}$. Multiplying by $h^{-1}$, we get
    \begin{equation}\label{ProofStep2.2}
        0 \ \leq \ \E\bigg[\mathbbm{1}_{\{{\tau_h}<\tau'_h\}}\frac{1}{h}\int_{t}^{\tau_h} f(s,X^{t,x}_s)ds + \mathbbm{1}_{\{{\tau_h}\geq \tau'_h\}}\frac{1}{h}\int_t^{\tau'_h}\big(\partial_t\varphi + \Lc_t\varphi + f\big)(s,X^{t,x}_s)ds\bigg] + h.
    \end{equation}
    Letting $h\rightarrow0^+$, we find
    \[
        \mathbbm{1}_{\{{\tau_h}\geq \tau'_h\}}\frac{1}{h}\int_t^{\tau'_h}\big(\partial_t\varphi + \Lc_t\varphi + f\big)(s,X^{t,x}_s)ds \ \overset{h\to0^+}{\longrightarrow} \ \big(\partial_t\varphi + \Lc_t\varphi + f\big)(t,x)\quad\text{a.s.}
    \]
    Moreover, we have
    \begin{align*}
        \bigg|\mathbbm{1}_{\{{\tau_h}<\tau'_h\}}\frac{1}{h}\int_{t}^{\tau_h} f(s,X^{t,x}_s)ds\bigg|\ &\leq \ \mathbbm{1}_{\{{\tau_h}<\tau'_h\}}\frac{1}{h}\int_{t}^{\tau_h} |f(s,X^{t,x}_s)|ds\\
        &\leq \ \mathbbm{1}_{\{{\tau_h}<\tau'_h\}}\max_{s\in[t,T]}|f(s,X^{t,x}_s)|\frac{\tau_h-t}{h}\\
        &\leq \ \mathbbm{1}_{\{{\tau_h}<\tau'_h\}}\max_{s\in[t,T]}|f(s,X^{t,x}_s)|\longrightarrow0\qquad\text{a.s.}\qquad\text{as }h\to0^+.
    \end{align*}
   In particular, the last inequality is due to the fact that $\tau_h-t<h$ on $\{{\tau_h}<\tau'_h\}$. Hence, letting $h\rightarrow0^+$ in \eqref{ProofStep2.2}, we conclude that
    \[
        0 \ \leq \ \big(\partial_t\varphi + \Lc_t\varphi + f\big)(t,x).
    \]
\end{proof}

\section{Approximation result for stopping times}\label{S:Approx}

Let $(\Omega,\Fc,\P)$ be a complete probability space and let $T>0$. Given $t\in[0,T]$, we consider a filtration $\mathbb H=(\Hc_s)_{s\in[t,T]}$ and a $\sigma$-algebra $\Gc$, with $\Gc$ and $\Hc_T$ being independent. Let also $\mathbb F=(\Fc_s)_{s\in[t,T]}$ be defined as
\[
\Fc_s \ = \ \Gc\vee\Hc_s, \qquad s\in[t,T].
\]
We denote by $\Tc_{t,T}$ and $\tilde\Tc_{t,T}$ the set of all $[t,T]$-valued stopping times with respect to $\mathbb F$ and $\mathbb H$, respectively. We can now state the claimed approximation result.

\begin{Theorem}\label{T:Approx}
Let $\tau \in\Tc_{t,T}$. There exists a sequence $(N_n)_n\subset\N$ and a sequence of stopping times $(\tau_n)_n\subset\Tc_{t,T}$ converging to $\tau$ almost surely and such that, for every $n$,
    \begin{equation}\label{successione tau}
        \tau_n = \sum_{i=1}^{N_n}\tau_{n,i}\mathbbm{1}_{G_{n,i}}, 
    \end{equation}
     where $\tau_{n,i}\in\tilde\Tc_{t,T}$ for all $i=1,\ldots,N_n$ and $\{G_{n,1},\ldots,G_{n,N_n}\}\subset\Gc$ is a partition of $\Omega$.
\end{Theorem}
\begin{proof}
    Since the subset of $\Tc_{t,T}$ consisting in stopping times that assume a finite number of values is dense in $\Tc_{t,T}$ (see for instance \cite[Lemma 3.3]{Baldi}), it suffices to consider stopping times $\tau\in\Tc_{t,T}$ of the form
    \begin{equation}\label{tau}
        \tau = \sum_{j=1}^M t_j\mathbbm{1}_{A_j},
    \end{equation}
    where $t_j\in[t,T]$, $A_j\in\Fc_{t_j}$ for all $j$ and $\{A_1,...,A_M\}$ is a partition of $\Omega$. We proceed in four steps.

    \vspace{2mm}
    
    \noindent\textsc{Step} 1. \emph{Fix $\tau\in\Tc_{t,T}$ as in \eqref{tau}. For every  $s\in[t,T]$ define the collection
    \[
        \Ac_s := \bigg\{A \in\Fc_s \big| \ A = \bigcup\limits_{i=1}^{N}\big(B_i\cap C_i\big),\, B_i\in\Gc,\,C_i\in\Hc_s, N\in\N\bigg\}. 
    \]
    We claim that for every $A_j\in\Fc_{t_j}$ there exists a sequence $(A_{n,j})_n\subset\Fc_{t_j}$ such that}
    \begin{equation}\label{Step1ClaimStoppTime}
        \tau_n \ \coloneqq \ \sum_{j=1}^M t_i\mathbbm{1}_{A_{n,j}} \ \longrightarrow \ \tau,\qquad\P\text{-a.s.}\qquad\text{as }n\to+\infty.
    \end{equation}
    It is easy to see that $\Ac_s$ is an algebra which generates $\Fc_s=\Gc\vee\Hc_s$. As a consequence, for every $A\in\Fc_s$ there exists $(A_n)_n\subset\Ac_s$ such that $\lim\limits_{n\to\infty}\P(A\Delta A_n) = 0$ (\cite[Theorem D, Sec. 13]{Halmos}). In particular, for every $A_j\in\Fc_{t_j}$ there exists a sequence $(A_{n,j})_n\subset\Fc_{t_j}$ such that $\lim\limits_{m\to\infty}\P(A_j\Delta A_{n,j}) = 0$, so that
    \[
        \E[|\mathbbm{1}_{A_{n,j}}-\mathbbm{1}_{A_j}|] \ = \  \E[\mathbbm{1}_{A_j\Delta A_{n,j}}] \ = \ \P(A_j\Delta A_{n,j}) \ \longrightarrow \ 0, \qquad\text{as }n\to+\infty.
    \]
    Then, up to a subsequence,
    \[
        \mathbbm{1}_{A_{n,j}} \ \longrightarrow \ \mathbbm{1}_{A_j},\qquad\P\text{-a.s.}\qquad\text{as }n\to+\infty.
    \]
    Since the latter holds for every $j\in\{1,...,M\}$, we see that \eqref{Step1ClaimStoppTime} follows.

    \vspace{2mm}
    
    \noindent\textsc{Step} 2. \emph{Without loss of generality, we can assume that for each fixed $n\in\N$ the events $A_{n,1},...,A_{n,M}$ form a partition of $\Omega$.} To prove this claim, let $n\in\N$ be fixed. From now on, we will assume, without loss of generality, that ${t\leq t_1<t_2<...<t_M\leq T}$, so that $\Ac_{t_j}\subset\Ac_{t_{j+1}}$. Thus, given
    \[
    \begin{cases}
        \Tilde{A}_{n,1} := A_{n,1}\\
        \Tilde{A}_{n,2} := A_{n,2}\setminus A_{n,1}\\
        \Tilde{A}_{n,j+1} := A_{n,j+1}\setminus\big(\cup_{k=1}^j A_{n,k}\big),&\text{for all } j=2,...,M-1
    \end{cases}
    \]
    we have that $\tilde A_{n,j}\in\Ac_{t_j}$ for every $j$. Moreover, since
    \[
        \mathbbm{1}_{\cup_{k=1}^j A_{n,k}} = \max_{1\leq k\leq j}\mathbbm{1}_{A_{n,k}},\qquad \mathbbm{1}_{\cup_{k=1}^j A_{k}} = \max_{1\leq k\leq j}\mathbbm{1}_{A_{k}}
    \]
    and, for every $k$,
    \[
        \mathbbm{1}_{A_{n,k}}\longrightarrow\mathbbm{1}_{A_k}\qquad\P\text{-a.s.}\qquad\text{as }n\to+\infty,
    \]
    by \textsc{Step} 1, it holds that 
    \begin{equation}\label{conv unions}
         \mathbbm{1}_{\cup_{k=1}^j A_{n,k}} \ \longrightarrow \ \mathbbm{1}_{\cup_{k=1}^j A_{k}}\qquad\P\text{-a.s.}\qquad\text{as }n\to+\infty.
    \end{equation}
    Then, noting that
    \[
        \mathbbm{1}_{ \tilde A_{n,j}} = \mathbbm{1}_{A_{n,j}\setminus\cup_{k=1}^{j-1}A_{n,k}}=\mathbbm{1}_{A_{n,j}}\big(1-\mathbbm{1}_{\cup_{k=1}^{j-1}A_{n,k}}\big),
    \]
    we get
    \[
         \mathbbm{1}_{ \tilde A_{n,j}} \ \longrightarrow \ \mathbbm{1}_{A_j\setminus\cup_{k=1}^{j-1}A_{k}}=\mathbbm{1}_{A_j}\qquad\P\text{-a.s.}\qquad\text{as }n\to+\infty,
    \]
    where the last equality holds because, by assumption, the events $A_j$ are disjoint. Hence
    \[
        \sum_{j=1}^M t_j\mathbbm{1}_{\tilde A_{n,j}} \ \longrightarrow \ \tau\qquad\P\text{-a.s.}\qquad\text{as }n\to+\infty.
    \]
    Now define
    \[
        \tilde A_{n,M+1} \coloneqq \bigg(\bigcup_{j=1}^M A_{n,j}\bigg)^\mathrm{c}\in\Ac_{t_M},
    \]
    so that the collection of events $(\tilde A_{n,j})_{j=1}^{M+1}$ constitutes a partition of $\Omega$. As in \eqref{conv unions}, we have that
    \[
        \mathbbm{1}_{\tilde A_{n,M+1}}\longrightarrow\mathbbm{1}_{(\cup_{j=1}^M A_{j})^\mathrm{c}} = \mathbbm{1}_{\Omega^\mathrm{c}} = \mathbbm{1}_{\emptyset} = 0,\qquad\P\text{-a.s.}\qquad\text{as }n\to+\infty,
    \]
    where the equalities hold because $(A_j)_j$ is a partition of $\Omega$. Consequently,
    \[
        \sum_{j=1}^{M} t_j\mathbbm{1}_{\tilde A_{n,j}} + t_M\mathbbm{1}_{\tilde A_{n,M+1}} \ \longrightarrow \ \tau\qquad\P\text{-a.s.}\qquad\text{as }n\to+\infty.
    \]
    These computations implies that, without loss of generality, we can assume that for each fixed $n\in\N$ the events $A_{n,1},...,A_{n,M}$ form a partition of $\Omega$. Therefore, it remains to show that $\tau_n=\sum_{j=1}^M t_j\mathbbm{1}_{A_{n,j}}$ has the desired form given by \eqref{successione tau}.
    

    \vspace{2mm}
    
    \noindent\textsc{Step} 3. \emph{Let $n\in\N$ and $j\in\{1,...,M\}$ be fixed and consider the event $A_{n,j}\in\mathcal{A}_{t_j}$ whose form is given by
    \[
        A_{n,j} = \bigcup_{i=1}^{N_{n,j}} (B_{n,j,i}\cap C_{n,j,i}),
    \]
    where $N_{n,j}\in\N$, $B_{n,j,i}\in\Gc$ and $C_{n,j,i}\in\Hc_{t_j}$ for every $i$. We claim that we can always assume the events $B_{n,j,i}$ to be disjoint with respect to the index $i$ in the decomposition of $A_{n,j}$.} To this end, let us define the following objects:
    \begin{itemize}
        \item $S_j\coloneqq\big\{\sigma = \{\sigma(1),...,\sigma(N_{n,j})\}: \sigma(i)\in\{0,1\}\,\forall i, \sigma\neq\{0,...,0\}\big\}$;
        \item $\hat{B}_{n,j,i}^k\coloneqq
        \begin{cases}
            B_{n,j,i}^c,& k=0\\
            B_{n,j,i},& k=1
        \end{cases},\qquad$ for every $i=1,...,N_{n,j}$;
        \item $\hat{C}_{n,j,i}^k\coloneqq
        \begin{cases}
            \emptyset,& k=0\\
            C_{n,j,i},& k=1
        \end{cases},\qquad$ for every $i=1,...,N_{n,j}$.
    \end{itemize}
    Firstly, we prove that the collection $\big\{\hat{B}_{n,j}^\sigma\coloneqq\bigcap_{i=1}^{N_{n,j}}\hat{B}_{n,j,i}^{\sigma(i)}\big\}_{\sigma\in S_j}$ is a partition of $\bigcup_{i=1}^{N_{n,j}} B_{n,j,i}$. By definition the events $\hat{B}_{n,j}^\sigma$ are pairwise disjoint with respect to $\sigma\in S_j$.  Indeed, if $\sigma,\,\rho\in S_j$ are such that $\sigma\neq \rho$, then $\sigma(\bar{i})\neq \rho(\bar{i})$ for some $\bar{i}\in\{1,...,N_{n,j}\}$. Assume that $\sigma(\bar{i})=1$, and so $\rho(\bar{i})=0$. In this case we have 
    \[
        \hat{B}_{n,j}^\sigma \ \subseteq \ \hat{B}_{n,j,\bar{i}}^{\sigma(\bar{i})} \ = \ B_{n,j,\bar{i}},\qquad \hat{B}_{n,j}^\rho \ \subseteq \ \hat{B}_{n,j,\bar{i}}^{\rho(\bar{i})}\ = \ B_{n,j,\bar{i}}^c,
    \]
    which implies that $\hat{B}_{n,j}^\sigma$ and $\hat{B}_{n,j}^\rho$ are disjoint.
    Similarly, if $\sigma(\bar{i})=0$, and so $\rho(\bar{i})=1$, the two events are also disjoint. Then, we only need to show that 
    \[
         \bigcup_{\sigma\in S_j}\hat{B}_{n,j}^\sigma \ = \ \bigcup_{i=1}^{N_{n,j}} B_{n,j,i}.
    \]
    This can be easily verified by finite induction on $N_{n,j}$.
    Secondly, it is straightforward to show that 
        \[
                A_{n,j} \ = \ \bigcup_{\sigma\in S_j}\bigg(\bigcap_{i=1}^{N_{n,j}} \hat{B}_{n,j,i}^{\sigma(i)}\bigg)\cap\bigg(\bigcup_{i=1}^{N_{n,j}} \hat{C}_{n,j,i}^{\sigma(i)}\bigg) \ \eqqcolon \ \bigcup_{\sigma\in S_j}\big( \hat{B}_{n,j}^\sigma \cap \hat{C}_{n,j}^\sigma\big).
        \]
    To summarize, we have demonstrated that each  $A_{n,j}$ can be represented as the finite disjoint union
    \[
        A_{n,j} \ = \ \bigcup_{\sigma\in S_j} (\hat{B}_{n,j}^\sigma \cap \hat{C}_{n,j}^\sigma),
    \]
    where $\hat{B}_{n,j}^\sigma\in\Gc$, $\hat{C}_{n,j}^\sigma\in\Hc_{t_j}$ for all $\sigma\in S_j$. Then, if $\{A_{n,j}\}_{j=1}^M$ forms a partition of $\Omega$, then the collection $\{\hat{B}_{n,j}^\sigma \cap \hat{C}_{n,j}^\sigma\}_{\sigma\in S}$, where $S\coloneqq \cup_j S_j$, also constitutes a partition.

    \vspace{2mm}
    
    \noindent\textsc{Step} 4. \emph{We claim that  $\tau_n=\sum_{j=1}^M t_j\mathbbm{1}_{A_{n,j}}$ has the desired form given by \eqref{successione tau}.} By \textsc{Step} 3, without loss of generality, we can always assume that every $A_{n,j}\in\mathcal{A}_{t_j}$ is of the form
    \[
        A_{n,j} \ = \ B_{n,j}\cap C_{n,j},
    \]
    where $B_{n,j}\in\Gc$, $C_{n,j}\in\Hc_{t_j}$ and $j\in\{1,...,M\}$. Now, for every $j\in\{1,...,M\}$, define the following objects:
    \begin{itemize}
        \item $\Sigma_j \coloneqq \big\{\sigma =\{\sigma(1),...,\sigma(M)\}:\sigma(i)\in\{0,1\}\,\, \forall i\in\{1,...,M\},\, \sigma(j)=1\big\}$;
        \item $\hat{B}_{n,j}^k \coloneqq \begin{cases}
            B_{n,j}^c,& k=0\\
            B_{n,j},& k=1
        \end{cases}$;
        \item $\hat{C}_{n,j}^k \coloneqq \begin{cases}
            \emptyset,& k=0\\
            C_{n,j},& k=1
        \end{cases}$.
    \end{itemize}
    Then, it is easy to verify that we can rewrite the events $A_{n,j}$ as follows:
    \[
    \begin{cases}
        A_{n,1} = \bigcup_{\sigma\in\Sigma_1} \big[\big(\bigcap_{i=1}^M \hat{B}_{n,i}^{\sigma(i)}\big)\cap C_{n,1}\big],\\
        A_{n,j} = \bigcup_{\sigma\in\Sigma_j} \big[\big(\bigcap_{i=1}^M \hat{B}_{n,i}^{\sigma(i)}\big)\cap \big(C_{n,j}\setminus \cup_{i<j}\hat{C}_{n,i}^{\sigma(i)}\big)\big],& j\geq 2.
    \end{cases}
    \]
    Now, let us define the following new objects: 
    \begin{itemize}
        \item $\Sigma \coloneqq \cup_{j=1}^M\Sigma_j = S_M$, where $S_M$ is the family of sequences defined in \textsc{Step} 3;
        \item $\hat{B}_n^\sigma \coloneqq \bigcap_{i=1}^M \hat{B}_{n,i}^{\sigma(i)}$ for every $\sigma\in\Sigma$;
        \item $\hat{C}_{n,1}^\sigma \coloneqq
        \begin{cases}
        \emptyset,&\sigma\in\Sigma\setminus\Sigma_1\\
        C_{n,1},&\sigma\in\Sigma_1
        \end{cases},\hspace{2.6cm}$ for $j=1$,\\
        $\hat{C}_{n,j}^\sigma \coloneqq 
        \begin{cases}
            \emptyset,& \sigma\in\Sigma\setminus\Sigma_j\\
            C_{n,j}\setminus \big(\cup_{i<j}\hat{C}_{n,i}^{\sigma(i)}\big),&\sigma\in\Sigma_j
        \end{cases},\quad$ for $1<j<M$,\\
        $\hat{C}_{n,M}^\sigma \coloneqq \bigg(\bigcup_{i<M} \hat{C}_{n,i}^{\sigma(i)}\bigg)^c,\hspace{3.3cm}$ for $j=M$.
    \end{itemize}
    Then, it is straightforward to prove that we can represent the events $A_{n,j}$ by
    \begin{equation}\label{repr events A}
        A_{n,j} = \bigcup_{\sigma\in\Sigma} \hat{B}_n^\sigma \cap \hat{C}_{n,j}^\sigma,\qquad\text{for every }j=1,...,M.
    \end{equation}
    Thanks to the decomposition \eqref{repr events A} of the events $A_{n,j}$, it holds that
    \begin{align*}
        \sum_{j=1}^M t_j\mathbbm{1}_{A_{n,j}} \ &= \ \sum_{j=1}^M t_j\mathbbm{1}_{\cup_{\sigma\in \Sigma} (\hat{B}^\sigma_{n}\cap \hat{C}^\sigma_{n,j})} \ = \ \sum_{j=1}^M\bigg(\sum_{\sigma\in \Sigma} t_j\mathbbm{1}_{\hat{C}^\sigma_{n,j}}\mathbbm{1}_{\hat{B}^\sigma_{n}}\bigg)\\
        &= \ \sum_{\sigma\in \Sigma}\bigg(\sum_{j=1}^M t_j\mathbbm{1}_{\hat{C}^\sigma_{n,j}}\bigg)\mathbbm{1}_{\hat{B}^\sigma_{n}}\ \eqqcolon \ \sum_{\sigma\in \Sigma}\tau_n^\sigma \mathbbm{1}_{\hat{B}^\sigma_{n}}.
    \end{align*}
    Finally, we observe that for every $\sigma\in\Sigma$, since by construction $\hat{C}^\sigma_{n,j}\in\mathcal{H}_{t_j}$  for every $j$ and the collection $\{\hat{C}^\sigma_{n,j}\}_{j=1}^M$ is a partition of $\Omega$, $\tau_n^\sigma\in\tilde\Tc_{t,T}$. Then, we conclude that the sequence of stopping times
    \[
        \tau_n \ = \ \sum_{j=1}^M t_j\mathbbm{1}_{A_{n,j}} \ = \ \sum_{\sigma\in \Sigma}\tau_n^\sigma \mathbbm{1}_{\hat{B}^\sigma_{n}}
    \]
    has the desired form \eqref{successione tau}.
\end{proof}

\small
\bibliographystyle{plain}
\bibliography{bibliography.bib}

\end{document}